\documentclass{article}
\usepackage{amsmath}
\usepackage{amsthm}
\usepackage{amsfonts}
\usepackage{amssymb}
\usepackage{tikz-cd}
\usepackage[numbers]{natbib}
\usepackage{mathrsfs}
\usepackage{lipsum,hyperref}
\usepackage{enumerate}
\hypersetup{colorlinks=true, linkcolor=black,urlcolor=black, citecolor=black}
\let\amsamp=&
\makeatletter
\newcommand{\colim@}[2]{%
  \vtop{\m@th\ialign{##\cr
    \hfil$#1\operator@font colim$\hfil\cr
    \noalign{\nointerlineskip\kern1.5\ex@}#2\cr
    \noalign{\nointerlineskip\kern-\ex@}\cr}}%
}
\newcommand{\colim}{%
  \mathop{\mathpalette\colim@{\rightarrowfill@\scriptscriptstyle}}\nmlimits@
}
\renewcommand{\varprojlim}{%
  \mathop{\mathpalette\varlim@{\leftarrowfill@\scriptscriptstyle}}\nmlimits@
}
\renewcommand{\varinjlim}{%
  \mathop{\mathpalette\varlim@{\rightarrowfill@\scriptscriptstyle}}\nmlimits@
}
\makeatother
\newtheorem{thm}{Theorem}[section]
\newtheorem{lem}[thm]{Lemma}
\newtheorem{prop}[thm]{Proposition}
\newtheorem{cor}[thm]{Corollary}
\newtheorem{definition}[thm]{Definition}
\theoremstyle{remark}
\newtheorem{rem}[thm]{Remark}

\newtheoremstyle{gabber}
   {6pt}                    
    {0pt}                    
    {\normalfont}                   
    {}                           
    {\normalfont}                   
    {.}                          
    {0.5em}                       
    {}  
\theoremstyle{gabber}
\newtheorem{sect}[thm]{}

\newcommand{\im}{\text{Im}}

\renewcommand{\ker}{\text{Ker}}

\newcommand{\ideal}[1]{\mathfrak{#1}}
\renewcommand{\lim}{\varprojlim}
\newcommand{\inj}{\colim}

\renewcommand{\cal}[1]{\mathcal{#1}}
\renewcommand{\mapsto}{\longmapsto}
\renewcommand{\to}{\longrightarrow}
\newcommand{\spec}{\text{Spec}}
\newcommand{\coker}{\text{Coker}}
\newcommand{\tor}[4]{\text{Tor}_{#1}^{#2}(#3, #4)}
\newcommand{\ext}[4]{\text{Ext}_{#2}^{#1}(#3, #4)}
\newcommand{\alext}[4]{\text{alExt}_{#2}^{#1}(#3, #4)}
\renewcommand{\hom}[3]{\text{Hom}_{#1}(#2, #3)}
\newcommand{\alhom}[3]{\text{alHom}_{#1}(#2, #3)}
\renewcommand{\cal}[1]{\mathcal{#1}}

\renewcommand{\tilde}[1]{\widetilde{#1}}

\usepackage{authblk}
\numberwithin{equation}{subsection}
\title{Almost Witt vectors}
\author{Ivan Zelich}
\date{}
\begin{document}
\maketitle
\begin{abstract}
Our main goal in this paper is to prove results for Witt vectors in the almost category. We finish with an application to \textit{almost purity}, Theorem~\ref{thm:almpurchar0}.
\end{abstract}
\tableofcontents
\newpage
\section{Introduction} 
\begin{sect}Almost mathematics originated from the work of Faltings \cite{gerd} who took inspiration from Tate's work on ramification theory \cite{tatepdiv}. Faltings considered rings such $A_0 := \mathbb{Z}_p[T_1,...,T_n]$ and ramified extensions $A_k := \mathbb{Z}_p[p^{1/p^{k}}, T_1^{1/p^{k}}, ..., T_n^{1/p^{k}}]$. If we take a finite normal $A_0$-algebra $B_0$, supposing that it is finite \'{e}tale after inverting $p$, then in general $B_0$ could have some ramification in the special fibre above $p$. Motivated by Zariski-Negata purity, Faltings then considered the localised situation: $A_0':=\mathbb{Z}_p[T_1^{\pm 1},...,T_n^{\pm 1}]$ and $B_0'$ a finite normal algebra over $A_0'$ that is finite etale along the generic fibre. After constructing $A'_k$ and $B'_k$ analogously, the point is that the infinite extension $B'_{\infty} := \cup_{k} B'_k$ will be almost unramified over $A'_{\infty}:= \cup_k A'_k$ with respect to $(p^{1/p^{\infty}})$, and thus almost finite \'{e}tale. This result has since been called an `almost purity' theorem, but in the way that we have framed it, it could be seen also as a variant of Abhyankar's lemma for wild ramification, in the sense that adjoining enough $p^{\text{th}}$-power roots kills \text{almost all} ramification.\end{sect}
\begin{sect}Let us now recall some main constructs for almost mathematics. First, one fixes an almost setting $(R, \ideal{m})$ consisting of a ring $R$ with an ideal $\ideal{m} \subset R$ where $\ideal{m}^2 = \ideal{m}$.\footnote{One seemingly artificial condition that can be given to $\ideal{m}$ is that $\tilde{\ideal{m}} := \ideal{m} \otimes_R \ideal{m}$ be an $R$-flat module, and as shown in \cite[Remark 2.1.4]{gabram}, this condition is preserved by base-change. However, the results of thesis do not rely on this hypothesis\textemdash see \ref{dercatalmmod} below. Condition (B) holds under this flatness condition \cite[Proposition 2.1.7]{gabram}.} In many situations, for any integer $k>1$, the $k^{\text{th}}$ power elements of $\ideal{m}$ generate $\ideal{m}$, which is called Condition (B) in \cite{gabram}\textemdash we will assume $\ideal{m}$ has this property. An $R$-module $N$ is called \textit{almost zero} if $\ideal{m}N = 0$, and one checks that this property is closed under extensions, thus ensuring that the subcategory of almost zero modules forms a Serre subcategory. Briefly, if we have an exact sequence:
\[\begin{tikzcd}
0 \arrow[r]& M_1 \arrow[r]& M \arrow[r]& M_2 \arrow[r] & 0
\end{tikzcd}\]
and $M$ is almost zero, then it's clear that $M_1$ and $M_2$ are almost zero too. For the converse, $\ideal{m}\cdot M$ is in $M_1$ since $M_2$ is almost zero, so $\ideal{m} \cdot M = \ideal{m}^2 \cdot M \subset \ideal{m} \cdot M_1= 0$ since $M_1$ is almost zero.\\
\indent We call a morphism $f: M\to N$ an \textit{almost isomorphism} if $\ideal{m} \cdot \ker(f) = \ideal{m} \cdot \coker(f) =0$. Set $\tilde{\ideal{m}}:= \ideal{m} \otimes_R \ideal{m}$. \end{sect}
\begin{lem}\label{lem:almzero}
\begin{enumerate}[(i)]
\item An $R$-module $M$ is almost zero if and only if $\ideal{m} \otimes_R M \simeq 0$.
\item $f: M \to N$ is an almost isomorphism if and only if the induced morphism $ \tilde{\ideal{m}} \otimes_R M \to \tilde{\ideal{m}} \otimes_R N$ is an isomorphism.
\end{enumerate}
\end{lem}
\begin{proof}See~\cite[Remark 2.1.4]{gabram}.
\end{proof}
\begin{rem}\label{rem:iter}
We note that the natural inclusion $\ideal{m} \hookrightarrow R$ is an almost isomorphism since the kernel and cokernel are trivially killed by $\ideal{m}$. As a direct result of (ii), it follows that $\ideal{m} \otimes_R \tilde{\ideal{m}} \to \tilde{\ideal{m}}$ is an isomorphism\textemdash in other words, repetitive tensoring $\ideal{m} \otimes_R \ideal{m} \otimes_R ...$ stops after the first iteration. Crucially, $\tilde{\ideal{m}} \otimes_R \tilde{\ideal{m}} \simeq \tilde{\ideal{m}}$.
\end{rem}
\begin{sect}Formally, the almost category $R^a\text{-Mod}$ is then the quotient of $R\text{-Mod}$ by the Serre subcategory of almost zero modules, which is an abelian category itself that is equipped with a localisation functor $(-)^a: R\text{-Mod} \to R^a\text{-Mod}$. \cite[02MN]{stacks} One typically constructs $R^a\text{-Mod}$ as the localised category at the multiplicative system of \text{almost isomorphisms}, and with this characterisation, we may understand the morphisms in $R^a\text{-Mod}$ via a calculus of fractions.\end{sect}
\begin{lem}\label{lem:initalm}
We have $\hom{R^a}{M^a}{N^a} = \hom{R}{\tilde{\ideal{m}} \otimes_R M}{N}$.
\end{lem} 
\begin{proof}
See~\cite[2.2.2]{gabram}.
\end{proof}
Combining this with Lemma~\ref{lem:almzero}, we see that if $f: M \to N$ is an almost isomorphism, then the induced morphism $f^a: M^a \to N^a$ is an isomorphism in the almost category. Moreover, an isomorphism $f: M^a \to N^a$in $R^a\text{-mod}$ can be represented uniquely by a morphism $\tilde{\ideal{m}} \otimes_R M \to \tilde{\ideal{m}} \otimes_R N$ in $R\text{-mod}$.\\
\indent There are two adjoints to $(-)^a$: for an almost module $M^a$ in $R^a\text{-mod}$, consider the following functors:
\begin{itemize} \item the functor of \textit{almost elements} $(-)_*: R^a\text{-mod} \to R\text{-mod}$ which takes $M^a \mapsto \hom{R}{\tilde{\ideal{m}}}{M}$. 
\item The functor $(-)_!: R^a\text{-mod} \to R\text{-mod}$ which takes $M^a \mapsto \tilde{\ideal{m}} \otimes_R M_*$. 
\end{itemize}
This is the right (resp. left adjoint) of $(-)^a$, and the counit (resp. unit) is a natual isomorphism in $R^a\text{-mod}$. \\
\indent There is a lot more foundational material for almost mathematics that can be found in \cite{gabram}, and we refer the reader there for more details. It may be useful for the reader to note while reading \cite[2.2]{gabram}, that due to Remark~\ref{rem:iter}, $\tilde{\ideal{m}}$ acts as a unit object on the essential image of $(-)_*$, so the theory of abelian tensor categories applies here, concretely. We will not discuss these foundations in this paper further, other than explaining key some key concepts.
\begin{sect}\label{fibalm} In particular, we would like a formalism for base-change in almost mathematics, which we can achieve by defining the following category $\cal{B}$ of `almost set-ups': objects are pairs $(R, \ideal{m})$ of almost set-ups, and morphisms $(R, \ideal{m}_R) \to (S, \ideal{m}_S)$ between two objects are ring homomorphisms $f: R \to S$ such that $\ideal{m}_S = f(\ideal{m}_R) \cdot S$.\footnote{This latter condition is precisely what will ensure that the almost structures after extension of scalars are compatible.} We have fibered and cofibered categories $\cal{B}\text{-Mod} \to \cal{B}$, where objects in the former category are pairs $((R, \ideal{m}_R), M)$ with $M$ an $R$-module and morphisms between $((R, \ideal{m}_R), M)$ and $((S, \ideal{m}_S), N)$ are pairs $(f,g)$ where $f: (R, \ideal{m}_R) \to (S, \ideal{m}_S)$ is a morphism in $\cal{B}$ and $g: M \to N$ is $f$-linear. We also denote $\cal{B}\text{-Alg}$ and $\cal{B}\text{-Mon}$ to be the category of algebras and non-unital monoids of $\cal{B}$, which are fibered/cofibered over $\cal{B}$.
\indent The almost isomorphisms in the fibers of $\cal{B}\text{-Mod} \to \cal{B}$ give a multiplicative system $\Sigma$ in $\cal{B}\text{-Mod}$. This allows us to form the localised category $\cal{B}^a\text{-Mod} := \Sigma^{-1}(\cal{B}\text{-Mod})$. The fibers of this localised category over the objects $(R, \ideal{m})$ are precisely the almost modules $R^a\text{-Mod}$, and similar assertions hold when restricting to $\cal{B}\text{-Alg}$ and $\cal{B}\text{-Mon}$. One can form adjoints to $(-)^a$ whose restrictions on the fibres induce the previously considered left and right adjoints $(-)_{!}$ and $(-)_*$.\\
\indent We can consider the category $\cal{B}/A$ defined to be the full subcategory of objects $(B, \ideal{m})$ where $B$ is an $A$-algebra. When $(R, \ideal{m})$ is an object of $\cal{B}/\mathbb{F}_p$, we can define a \textit{Frobenius} endomorphism $\Phi_R: (R,\ideal{m}) \to (R, \ideal{m})$ (which is a morphism in $\cal{B}$ as $\ideal{m}$ satisfies Condition (B)). For any $B \in \cal{B}$-$\text{Alg}/\mathbb{F}_p$ (resp. $\cal{B}$-$\text{Mon}/\mathbb{F}_p$), the Frobenius map induces a morphism $\Phi_B: B \to B$ over $\Phi_R$. The collection of Frobenius maps give us a natural transformation from the identity functor of $\cal{B}\text{-Alg}/\mathbb{F}_p$ (resp. $\cal{B}\text{-Mon}/\mathbb{F}_p$) to itself that induces a natural transformation on the identity functor of $\cal{B}^a\text{-Alg}/\mathbb{F}_p$ (resp. $\cal{B}\text{-Mon}/\mathbb{F}_p$). Using pull-back functors, any object of $\cal{B}\text{-Alg}$ over $R$ defines new objects $B_{(m)}$ of $\cal{B}\text{-Alg}$ $(m \in \mathbb{N})$ over $R$, where $B_{(m)} := (\Phi_R^m)_{*}(B)$.\end{sect}

\begin{sect}The main outcome of this thesis was to find an alternative proof of almost purity for perfectoid valuation rings of rank one, without relying on too much analytic contexts, ultimately in the hope of potentially extending the proof to the general case. Throughout this discussion, we fix a perfectoid ring $R$ (in mixed characteristic) with an element $\varpi \in R$ admitting a compatible system of $p$-power roots \cite[Lemma 3.2]{perfectoid}; let us recall what we mean by this:\end{sect}
\begin{definition}
A complete, non-discrete valuation ring $R$ of rank $1$ is a \textit{perfectoid valuation ring} if the Frobenius $\Phi:R/pR \twoheadrightarrow R/pR$ is surjective.
\end{definition}
\noindent Using the theory of perfectoid spaces, one can prove almost purity for arbitrary perfectoid rings via descent to perfectoid valuation rings of rank one, using the fact that the fibered category of finite \'{e}tale covers of adic spaces forms a stack~\cite[Corollary 15.7.26]{gabram2}, as done in Scholze's thesis \cite{perfectoid}. Let us finally state the theorem we want to prove:
\begin{thm}[Almost purity in Characteristic 0]\label{thminto:almpurchar0}
Fix a perfectoid valuation ring $R$ of rank one. Let $S$ be a finitely presented $R$-algebra, and suppose $S[1/p]$ is a finite \'{e}tale algebra over $R[1/p]$. Then $\cal{O}_S$, the integral closure of $R$ in $S[1/p]$, is almost finitely presented \'{e}tale over $R$.
\end{thm}
\noindent Our approach is to reframe the usual `untilting' functor from characteristic $p$ to characteristic $0$ via finite length Witt vectors. In particular, we prove the following (see Theorem~\ref{thm:wittdescent2}):
\begin{thm}\label{thminto:wittdescent}
Let $A \to B$ be a weakly \'{e}tale flat morphism of $R^a$-algebras. Then,  $W_n(f): W_n(A) \to W_n(B)$ is weakly \'{e}tale.
\end{thm}
\noindent See \ref{def:weaketale} for the appropriate definitions. Armed with this theorem, the following result isn't difficult (see Lemma~\ref{lem:absabh}):
\begin{thm}
Let $\overline{R}$ be the absolute closure of a perfectoid valuation ring $R$ of rank one. Then $\overline{R}/p^n$ is almost weakly \'{e}tale $(R/p^n, (\varpi^{1/p^{\infty}}))$-algebra for all $n \ge 1$.
\end{thm}
\noindent Almost purity then follows via descent arguments.\\
\begin{ackn}
The author is incredibly grateful to Dr James Borger and Dr Lance Gurney for the insightful discussions, support, and supervision. The author was funded by an Australian Government Research Training Program Scholarship (2020-2021) during which this paper was completed for partial fulfillment of an MPhil degree.
\end{ackn}
\section{Technical results}
\subsection{Derived categories of almost modules}\label{dercatalmmod}Here we shall extend some results from \cite[Section 2.4]{gabram} by removing the hypothesis that $\tilde{\ideal{m}}$ be a flat $R$-module. We note also that more is accomplished towards this aim in~\cite[14.1]{gabram2}.
\begin{definition}
Let $A$ be an $R^a$-algebra, and $M$ an $A$-module.
\begin{enumerate}[(i)]
\item We say that $M$ is flat (resp. faithfully flat) if the functor $N \mapsto M \otimes_A N$, from the category of $A$-modules to itself is exact (resp. exact and faithful).
\item We say that $M$ is almost projective if the functor $N \mapsto \alhom{A}{M}{N}$ is exact.
\end{enumerate}
\end{definition}
\begin{lem}\label{lem:abcatprop}
Let $A$ be an $R^a$-algebra.
\begin{enumerate}[(i)]
\item $A\text{-Mod}$ and $A\text{-Alg}$ are both complete and cocomplete, with exact colimits.
\item $(-)^a: A_*\text{-Mod} \to A\text{-Mod}$ preserves flat (resp. faithfully flat) $A$-modules, and sends projective objects to almost projective objects.
\end{enumerate}
\end{lem}
\begin{proof}
(i): We note that $A_*\text{-Mod}$ and $A_*\text{-Alg}$ are both complete and cocomplete. We show that $A\text{-Mod}$ is cocomplete, and the other assertions follow similarly. Suppose $I$ is a small indexing category and set $M := \inj_{i \in I} M(i)_*$; we see $M^a := \inj_{i \in I} M(i)$ as $(-)^a$ commutes with colimits (and limits) and the unit $(-)_*^a $ of the corresponding adjunction is naturally isomorphic to the identity.\\
\indent Note that this argument says that colimits are left exact since $(-)_*$ is; so by repeating the same argument with instead $(-)_{!}$ we get that colimits are right exact too, whence the assertion.\\
\indent (ii): Suppose that $M$ is a flat $A_*$-module i.e. the endofunctor $M \otimes_{A_*} -$ is exact. Since the aforementioned functor preserves the Serre subcategory of almost zero $A_*$-modules, $M^a \otimes_A -$ will too be an exact endofunctor on $A\text{-Mod}$. For the faithfully flat assertion, we only must show that if $M^a \otimes_{A} N \simeq 0$ then $N \simeq 0$. We're given that $M \otimes_{A_*} N_*$ is almost zero, so that $(M \otimes_{A_*} N_*) \otimes_{A_*} \ideal{m}A_* \simeq 0$. By faithful flatness, $N_* \otimes_{A_*} \ideal{m}A_* \simeq 0$, so that $N_*$ is almost zero, or indeed $N \simeq (N_*)^a \simeq 0$. The assertion for projective objects is clear by definition.
\end{proof}
Now let $B$ be any $R$-algebra; for any interval $I \subset \mathbb{N}$, we may extend $(-)^a$ termwise to a functor:
\[(-)^a: C^I(B\text{-Mod}) \to C^I(B^a\text{-Mod}).\]
We may consider the class of maps $\Sigma$ as those for which the modules $H^i(\text{Cone}(\phi))$ are almost zero. Then by the exactness of $(-)^a$, we see that it descends to a functor:
\[(-)^a: \Sigma^{-1}D^I(B\text{-Mod}) \to D^I(B^a\text{-Mod}).\]
We're interested in having a right/left adjoint on these derived categories:
\begin{lem}\label{lem:deralmloc}
The functor $(-)^a: \Sigma^{-1}D^I(B\text{-Mod}) \to D^I(B^a\text{-Mod})$ is an equivalence of categories with quasi-inverse $(-)_{!}$ or $(-)_*$.
\end{lem}
\begin{proof}
Extend $(-)_{*}$ termwise to a functor $(-)_{*}: C^I(B^a\text{-Mod}) \to C^I(B\text{-Mod})$. We note then that by definition $(-)_{*}$ does indeed descend to a functor $D^I(B^a\text{-Mod}) \to \Sigma^{-1} D^I(B\text{-Mod})$ such that $(-)_{*}^a: D^I(B^a\text{-Mod}) \to D^I(B^a\text{-Mod})$ is naturally isomorphic to the identity.\footnote{It is not necessarily true that $(-)_{*}$ is a functor with target $D^I(B\text{-Mod})$, because it is only left-exact in general.} Now by construction $(-)^a_{*}: \Sigma^{-1} D^I(B\text{-Mod}) \to \Sigma^{-1} D^I(B\text{-Mod})$ is too naturally isomorphic to the identity, as desired. The same argument applies for $(-)_!$.
\end{proof}
By Lemma~\ref{lem:abcatprop}, we see that $A\text{-Mod}$ satisfies the AB5 axiom for abelian categories, and as the category has a generator, namely $A$, we can conclude that $A\text{-Mod}$ has enough injectives. From the same lemma, we may also conclude that $A\text{-Mod}$ has enough flat/almost projective objects.\\
\indent Given an $A$-module $M$, we can derive the functors $M\otimes_A -$ (resp. $\alhom{A}{M}{-}$, resp. $\alhom{A}{-}{N}$) by taking flat (resp. almost projective, resp. injective) resolutions. Indeed, bounded above exact complexes of flat (resp. almost projective) modules are acyclic for the functor $M \otimes_A -$ (resp. $\alhom{A}{-}{N}$), which enables us to compute the derived functors like usual. Let us sketch why in the case of $\alhom{A}{-}{N}$: suppose $P_{\bullet}$ is a bounded above complex of projective objects, and $J^{\bullet}$ an injective resolution for $N$. Note that, as discussed (see Lemma~\ref{lem:initalm}), for any $A$-module $N$ and a complex $\mathscr{C}$ of $A_*$-modules, $\alhom{A}{\mathscr{C}^a}{N} =\hom{A_*}{\mathscr{C}}{N_*}^a$. Thus, by considering when $H^i(\mathscr{C})^a\simeq 0$ and using that $(-)_{*}$ preserves injectives, we see that $\alhom{A}{-}{J^i}$ is exact. Considering $C$, the double complex $C_{ij}:=\alhom{A}{P_{i}}{J^j}$, we see firstly that $\text{Tot}(C)$ is quasi-isomorphic to the complex $\alhom{A}{P_{\bullet}}{M}$ by the vertically filtered spectral sequence, and secondly that $\text{Tot}(C)$ is quasi-isomorphic to $0$ by the horizontally filtered spectral sequence. Thus, the class of almost projective objects is `adapted' to $\alhom{A}{-}{M}$, in the sense of \cite[4.3]{algebragel}, and similarly the flat objects will be adapted to $M \otimes_A -$. \cite[Theorem 4.8]{algebragel} or \cite[Theorem 10.5.9]{weibel} gives us a way of deriving these functors.\\
\indent Combining Lemmas~\ref{lem:abcatprop}, \ref{lem:deralmloc}, we see that for $A$-modules $S,T$:
\[\tor{n}{A}{S}{T} = \tor{n}{A_*}{S_*}{T_*}^a, \; \; \alext{n}{A}{S}{T} = \ext{n}{A_*}{S_*}{T_*}^a,\]
where alExt is the derived functor of alHom. Thus, an $A$-module $M$ is flat/almost projective if and only if the the higher Tor/Ext groups for the pair $(M_*,N_*)$ are almost zero, for any $R^a$-module $N$.
\begin{rem}
We briefly remark that, due to the fact that filtered colimits of flat objects are flat, Tor commutes with filtered colimits.
\end{rem}
\subsection{Flatness criteria}
We begin with a generalisation of a usual criterion for flatness in the almost category. In what follows, we will use the fact that any almost module can be written as a colimit of its finitely generated sub-modules, which follows from Lemma~\ref{lem:abcatprop} and the corresponding statement for modules. We fix throughout this section an almost set-up $(R, \ideal{m})$.
\begin{thm}\label{thm:finflatness}
Let $M$ be an $R^a$-module; $M$ is $R^a$-flat if and only if, for every finitely generated ideal $I \subset R^a$ we have $I \otimes_{R^a} M \to M$ is a monomorphism.
\end{thm}
\begin{proof}
The `only if' direction is clear. For the converse, note that as tensoring commutes with colimits and colimits are exact, we may assume the supposition for any ideal $I \subset R^a$. Given an exact sequence of $R^a$-module $\mathscr{E}:=0 \to N_1 \to N_2 \to N_3 \to 0$, we need to show that $\mathscr{E} \otimes_{R^a} M$ is exact, for which it suffices to show that $\tor{1}{R^a}{N_3}{M}=0$. Now, write $N_3$ as a colimit of its finitely generated $R^a$-modules; as Tor commutes with colimits, we reduce to the case when $N_3 = (R^{a})^{n}/L$ where $L \subset (R^a)^n$ an $R^a$-submodule. Next, consider the following diagram:
\[
\begin{tikzcd}
0 \arrow[r] & (R^{a})^{n-1} \arrow[r] & (R^a)^n \arrow[r] & R^a \arrow[r] & 0\\
0 \arrow[r] & L \times_{(R^a)^n} (R^{a})^{n-1} \arrow[r] \arrow[hookrightarrow]{u} & L \arrow[r] \arrow[hookrightarrow]{u} & I' \arrow[r] \arrow[hookrightarrow]{u} & 0
\end{tikzcd}\]
Here, $I'$ is the image of $L \subset (R^a)^n \to R^a$ (projection onto the last coordinate). After applying $-\otimes_{R^a} M$ to the diagram, we see that the bottom row remains right exact, and the top row remains exact as everything is free. By induction on $n$, we may assume that the left-up arrow is monomorphism, and by the hypothesis we know that the right-up arrow is. Snake lemma concludes.
\end{proof}
\begin{sect}\label{colimflat}Let $I$ be a (small) filtered category, and $F:I \to R^a\text{-Alg}$ a functor. Set $A_i := F(i)$, and for each $i \in I$, let $M_i, N_i$ be an $A_i$-modules. Set $A:= \colim_{i} A_i, M := \colim_i M_i$ and $N := \colim_i N_i$. \end{sect}
\begin{lem}\label{lem:colimflat}
In the situation of \ref{colimflat}:
\begin{enumerate}[(i)]
\item We have $\colim_{i} N_i \otimes_{A_i} M_i \simeq N \otimes_{A} M$.
\item If $M_i$ is $A_i$-flat, then $M$ is $A$-flat.
\end{enumerate}
\end{lem}
\begin{proof}
(i): By applying $(-)_{*}$ to both sides we reduce to the statement for ordinary modules via the same argument as in Lemma~\ref{lem:abcatprop} (i).\\
\indent (ii): Let $I \subset A$ be a finitely generated ideal. There is some index $i$ such that $I' \subset A_i$ and $I=I'A$; by Theorem~\ref{thm:finflatness}, we need to show $I \otimes_{A} M \to M$ is a monomorphism. Note that $I'A_j \otimes_{A_j} M_j \to M_j$ is monomorphism for each $j \ge i$, so by passing to the colimit and using (i), we get that $I \otimes_{A} M \to M$ is monomorphism by the exactness of colimits.
\end{proof}
\begin{cor}\label{cor:colimwet}
Let $A$ be an $R^a$-algebra and $B_i$ be weakly \'{e}tale algebras over $A$. Then, $\colim_{i} B_i$ is weakly \'{e}tale over $A$.
\end{cor}
\begin{proof}
Clearly, $A \to B$ is flat, so we need to show that it is weakly unramified. Indeed, we need to ensure $B \otimes_{A} B \to B$ is flat; we note $\colim_{i}B_i \otimes_{A} B_i  \simeq B \otimes_A B$ via Lemma~\ref{lem:colimflat} (i). Then by Lemma~\ref{lem:colimflat} (ii), we're done.
\end{proof}
We have the following almost version of the local flatness criterion, which we state more generally as follows:
\begin{thm}\label{thm:genlocflat}
Let $A$ be an $R^a$-algebra. Suppose that we are given morphisms $A \to A_i$ of (not necessarily distinct) $R^{a}$-algebras so that for every $A$-module $M$, one has a filtation $0 = F_{0}(M) \subset F_{1}(M) \subset ... \subset F_{n}(M) = M$ such that each graded piece $\text{gr}_{i}(M) := F_{i+1}(M)/F_i(M)$ is an $A_{i}$-module. Then the following are equivalent:
\begin{enumerate}[(i)]
\item $M$ is $A$-flat
\item $M_i := M \otimes_{A} A_i$ is $A_i$-flat and $\tor{1}{A}{A_i}{M} = 0$ for each $i=1,...,n-1$
\end{enumerate}
\end{thm}
\begin{proof}
We first begin with a lemma.
\begin{lem}\label{lem:filtorext}
In the situation above, to check that $\tor{1}{A}{M}{N}=0$ (or $\alext{1}{A}{M}{N}=0$) for arbitrary $A$-modules $N$, it is sufficient to do so for when $N$ is an $A_i$-module, for any index $i=1,...,n-1$.
\end{lem}
\begin{proof}
By the supposition, we know that each module $N$ adopts a filtration $F_{i}(N)$ such that $\text{gr}_{i}(N)$ is an $A_i$-module. Assuming $\tor{1}{A}{M}{K}=0$ for any $A_i$-module $K$, we conclude by using Tor sequences applied to the exact sequences:
\[\begin{tikzcd}
0 \arrow[r] & F_{i}(N) \arrow[r]& F_{i+1}(N) \arrow[r] & \text{gr}_{i}(N) \arrow[r] & 0
\end{tikzcd}\]
that $\tor{1}{A}{M}{F_{i}(N)}=0 \Rightarrow \tor{1}{A}{M}{F_{i+1}(N)}=0$. By definition, $F_{n}(N) := N$ and $F_{1}$ is an $A_0$-module, so we eventually get that $\tor{1}{A}{M}{N} = \tor{1}{A}{M}{F_1} = 0$ via an induction, as desired. The statement for alExt follows with minor changes.
\end{proof}
To get the result, we utilise the Tor base change spectral sequence; for each index $i=1,...,n-1$, and $A_i$-modules $N$:
\[E^{2}_{pq} := \tor{p}{A_i}{\tor{q}{A}{M}{A_i}}{N} \Rightarrow \tor{p+q}{A}{M}{N}.\]
Now, by assumption, along $q=1$ and $p=1$ we have $0$'s, which implies that $\tor{1}{A}{M}{N}=0$ for any $A_i$-module $N$. By the lemma, this is enough to prove the theorem.
\end{proof}
\begin{cor}[Local flatness criterion]\label{cor:almlocflat}
Suppose $A$ is an $R^a$-algebra, $M$ an $A$-module, and $I$ a nilpotent ideal. The following are equivalent:
\begin{enumerate}[(i)]
\item $M$ is $A$-flat
\item $M_0:= M/IM$ is $A_0 := A/I$-flat and $I^k/I^{k+1} \otimes_{A_0} M_0 \to I^kM/I^{k+1}M$ is an isomorphism for each $k\ge 1$.
\end{enumerate}
\end{cor}
\begin{proof}
We prove the non-obvious direction. Indeed, the previous theorem applies by considering the filtation $F_{n-k}(M) := I^kM$. So, we only need to show that $\tor{1}{A}{A/I}{M}=0$. Consider exact sequences:
\[\alpha_k := \begin{tikzcd} 0 \rar & \tor{1}{A}{A/I^k}{M} \rar & I^k \otimes_A M \rar & I^kM \rar & 0 \end{tikzcd}\]
Applying Snake lemma to $\alpha_{k+1} \to \alpha_k$, one obtains that 
\[\coker\left(\tor{1}{A}{A/I^{k+1}}{M} \to \tor{1}{A}{A/I^k}{M}\right) = 0.\]
Since $I$ is nilpotent, descending induction gives that $\tor{1}{A}{A/I}{M} = 0$
\end{proof}

\subsection{Almost homological algebra}
Here we collect some results on homological algebra in the almost category following \cite[Chapter 2.4]{gabram}. We note that Lemma~\ref{lem:deralmloc} and our subsequent results on deriving alHom and $\otimes$ in the almost category do not require the condition that $\tilde{\ideal{m}}$ be flat, but merely that $\ideal{m}^2=\ideal{m}$, in contrast to \textit{loc. cit}.
\begin{prop}Let $M$ be an $A$-module.
\begin{enumerate}[(i)]
\item $M$ is almost finitely generated if and only if for every generated ideal $\ideal{m}_0 \subset \ideal{m}$ there exists a finitely generated submodule $M_0 \subset M$ such that $\ideal{m}_0M \subset M_0$.
\item The following are equivalent:
\begin{enumerate}[(a)]
\item $M$ is almost finitely presented
\item for arbitrary $\epsilon, \delta \in \ideal{m}$, there exists positive integers $n = n(\epsilon), m = m(\epsilon)$ and a three term complex
\[\begin{tikzcd} A^m \arrow[r, "\psi_{\epsilon}"] & A^n \arrow[r, "\phi_{\epsilon}"] & M \end{tikzcd}\]
with $\epsilon \cdot \coker(\phi_{\epsilon})=0$ and $\delta \cdot \ker(\phi_{\epsilon}) \subset \im(\psi_{\epsilon})$. 
\item For every finitely generated ideal $\ideal{m}_0 \subset \ideal{m}$, there is a complex 
\[\begin{tikzcd} A^m \arrow[r, "\psi"] & A^n \arrow[r, "\phi"] & M \end{tikzcd}\]
with $\ideal{m}_0 \cdot \coker(\phi)=0$ and $\ideal{m}_0 \cdot \ker(\phi) \subset \im(\psi)$
\end{enumerate}
\end{enumerate}
\end{prop}
\begin{proof}
See \cite[Proposition 2.3.10]{gabram}. 
\end{proof}
\begin{lem}
Let $M$ be an almost finitely generated $A$-module and $B$ a flat $A$-algebra. Then $\text{Ann}_B(B \otimes_A M) \simeq B \otimes_A \text{Ann}_A(M)$.
\end{lem}
\begin{proof}
See~\cite[Lemma 2.4.6]{gabram}.
\end{proof}
\begin{lem}\label{lem:almprojsplit}
Let $M$ be an almost finitely generated $A$-module, $A$ an $R^a$-module. Then $M$ is almost finitely projective if and only if, for arbitrary $\epsilon \in \ideal{m}$, there exists $n(\epsilon) \in \mathbb{N}$ and $A$-linear morphisms
\[\begin{tikzcd}M \arrow[r, "u_{\epsilon}"] & A^{n(\epsilon)} \arrow[r, "v_{\epsilon}"] & M \end{tikzcd}\]
such that $v_{\epsilon} \circ u_{\epsilon} = \epsilon \cdot 1_{M}$.
\end{lem}
\begin{proof}
See~\cite[Lemma 2.4.15]{gabram}.
\end{proof}
We have the following:
\begin{thm}\label{thm:presented}
Let $A$ be an $R^a$-algebra. 
\begin{enumerate}[(i)]
\item Every almost finitely generated projective $A$-module is almost finitely presented.
\item Every almost finitely presented flat $A$-module is almost projective.
\end{enumerate}
\end{thm}
\begin{proof}
See~\cite[Proposition 2.4.18]{gabram}.
\end{proof}
We finish with some definitions:
\begin{definition}\label{def:weaketale}
\begin{enumerate}
\item We say that $\phi$ is flat (resp. fairthfully flat, resp. almost projective) if $B$ is a flat (resp. faithfully flat, resp. almost projective) A-module.
\item We say that $\phi$ is almost finite (resp. finite) if $B$ is almost finitely generated (resp. finitely generated) as an $A$-module.
\item We say that $\phi$ is weakly unramified (resp. unramified) if $B$ is a flat (resp. almost projective) $B \otimes_A B$-module (via the morphism $\mu_{B/A}$).
\item $\phi$ is weakly \'{e}tale (resp. \'{e}tale) if it is flat and weakly unramified (resp. unramified).
\end{enumerate}
\end{definition}
\begin{prop}\label{prop:assortlem}
Let $\phi: A \to B$ and $\psi: B \to C$ be morphisms of almost algebras.
\begin{enumerate}[(i)]
\item (Base change) Let $A \to A'$ be any morphism of $R^a$-algebras; if $\phi$ is flat (rep. almost projective, resp. faithfully flat, resp. almost finite, resp. weakly unramified, resp. unramified, resp. weakly \'{e}tale, resp. \'{e}tale), then the same holds for $\phi \otimes_A 1_{A'}$.
\item (Composition) If both $\phi, \psi$ are flat (resp. almost projective ...), then so is $\psi \circ \phi$.
\item  If $\phi$ is flat, and $\psi \circ \phi$ is faithfully flat, then $\phi$ is faithfully flat.
\item If $\phi$ is weakly unramified and $\psi \circ \phi$ is flat (resp. weakly \'{e}tale), then $\psi$ is flat (resp. weakly \'{e}tale).
\item If $\phi$ is unramified and $\psi \circ \phi$ is \'{e}tale, then $\psi$ is \'{e}tale.
\item $\phi$ is faithfully flat if and only if it is a monomorphism and $B/A$ is a flat $A$-module.
\item $\phi$ is almost finite and weakly unramified, then $\phi$ is unramified
\item If $\psi$ is faithfully flat and $\psi \circ \phi$ is flat (resp. weakly unramified), then $\phi$ is flat (resp. weakly unramified).
\end{enumerate}
\end{prop}
\begin{proof}
We detail the proof of (viii) as will need this later on, for the rest see \cite[Lemma 3.1.2.]{gabram}.\\
\indent (viii): If $\psi \circ \phi$ is flat, then we conclude by noting $-\otimes_A C = (-\otimes_A B)\otimes_{B} C$. As a consequence, we have that if $D$ is an $A$-algebra such that $A \to D$ is faithfully flat, then $B \to C$ is flat if and only if $B \otimes_A D \to C \otimes_A D$ is flat (we use (i) also). \\
\indent Assume $A \to C$ is weakly unramified, or indeed that $C \otimes_A C \to C$ is flat. Note that showing $B \otimes_A B \to B$ is flat is equivalent to showing that $B \otimes_A C \to C$ is flat, but by (i) and our supposition, we know that each map in the composition $B \otimes_A C \to C \otimes_A C \to C$ is flat, hence conclude by (ii).
\end{proof}
The following is an expected outcome of the definitions is crucial for applications to \textit{almost purity}, see Theorem~\ref{thm:almpurcharp}.
\begin{prop}
A morphism $\phi: A \to B$ is unramified if and only if there exists an almost element $e_{B/A} \in (B \otimes_A B)_*$ such that
\begin{enumerate}[(i)]
\item $e_{B/A}^2 = e_{B/A}$
\item $(\mu_{B/A})_*(e_{B/A})=1_{B}$
\item $x \cdot e_{B/A} = 0 $ for all $x \in I_{B_*/A_*}$
\end{enumerate}
\end{prop}
\begin{proof}
See \cite[Proposition 3.1.4]{gabram}.
\end{proof}
In what follows, we will call $B$ being almost finitely presented \'{e}tale if $B$ is an almost finitely presented $A$-module in addition to being \'{e}tale over $A$.

\begin{thm}\label{thm:frobpushout}
Let $f: A \to B$ be a weakly \'{e}tale morphism of $\mathbb{F}_p$-almost algebras. Then the diagram:
\[\begin{tikzcd}
B \arrow[r, "\Phi_B"] & B\\
A \arrow[u, "f"] \arrow[r, "\Phi_A"] & A \arrow[u, "f"]
\end{tikzcd}\]
is a push-out. In other words, the canonical morphism $g: B \otimes_{A, \Phi_A} A \to B_{\Phi_B}$ is an isomorphism of $A$-modules.
\end{thm}
\begin{proof}
One replicates the proof~\cite[Theorem 3.5.13]{gabram}, noting that it doesn't need the flatness conditions on $\tilde{\ideal{m}}$ that was assumed from the begining in Section 3.5, but rather that $\tilde{\ideal{m}}$ satisfies Condition (B), as we have assumed.
\end{proof}
\subsection{Descent}
In what follows, we have collected various results in \cite{gabram} necessary for us. Apart from some slight modifications in various proofs, the main contribution is the linearity in exposition.\\
\indent Given a morphism of almost $R^a$-algebras $A \to B$, we would like to ask under what conditions information about an $A$-module $M$ can be extracted from $M_B := M\otimes_A B$ as a $B$-module. We have the following:
\begin{thm}\label{thm:descent1}
Suppose that the morphism $A \to B$ satisfies the following property: There exists an integer $m\ge 0$ such that for any $A$-module $N$, $\text{Ann}_{A}(N_B)^m \subset \text{Ann}_{A}(N)$. (*)  Then for an $A$-module $M$:
\begin{enumerate}[(i)]
\item If $M_B$ is an almost finitely generated $B$-module, then $M$ is an almost finitely generated $A$-module.
\item If $\tor{A}{1}{B}{M}=0$ and $M_B$ is almost finitely presented over $B$, then $M$ is almost finitely presented over $A$.
\end{enumerate}
\end{thm}
The condition (*) imposed on $M$ at first glance might seem rather unnatural, but notice $M \otimes_A B \simeq 0 \Rightarrow A = \text{Ann}_{A}(M_B)^m \subseteq \text{Ann}_{A}(M) \subseteq A$ so $M \simeq 0$. Since we're in the almost setting, the condition (*) could be viewed as a modification of the condition that $M_B \simeq 0 \Longleftrightarrow M \simeq 0$ to deal with the `limiting' arguments that arise with almost finitely generated/almost finitely presented objects. The above discussion leads to:
\begin{cor}\label{cor:ffdescent}
Let $A \to B$ is faithfully flat and $M$ be an $A$-module.
\begin{enumerate}[(i)]
\item If $M_B$ is almost finitely generated or almost finitely presented over $B$, then $M$ has the same property over $A$.
\item If $M_B$ is almost finitely generated projective over $B$, then $M$ has the same property over $A$.
\item If $A \to B$ is also almost finitely presented as an $A$-module, then if $M_B$ is almost projective, $M$ has the same property over $A$.
\end{enumerate}
\end{cor}
\begin{proof}
\begin{enumerate}[(i)]
\item Follows from Theorem~\ref{thm:descent1} and the previous discussion.
\item From Theorem~\ref{thm:presented}, one concludes that $M_B$ is finitely presented, so from (i), $M$ is too. But $M$ is by definition flat, so by (ii) of the same theorem, $M$ is almost projective.
\item We know that $B$ is therefore almost finitely generated projective, so \cite[Lemma 2.4.31]{gabram} applies, which says that $B \otimes_A \alhom{A}{M}{N} \to \alhom{B}{B\otimes_A M}{ B \otimes_A N}$ is an isomorphism. Due to the faithful flatness of $B$, the proposition is evident.
\end{enumerate}
\end{proof}
For morphisms satisfying the descent condition (*) we have the following bounds:
\begin{lem}\label{lem:boundkercoker}
Assume that condition (*) holds. For any $A$-linear morphism $\phi: M \to N$, set $\phi_B := \phi \otimes_A 1_B$; then:
\begin{enumerate}[(i)]
\item $\text{Ann}_A(\coker(\phi_B))^m \subset \text{Ann}_A(\coker(\phi))$.
\item $(\text{Ann}_A(\ker(\phi_B)) \cdot \text{Ann}_A(\tor{1}{A}{B}{N}) \cdot \text{Ann}_A(\coker(\phi)))^m \subset \text{Ann}_A(\ker(\phi))$
\end{enumerate}
\end{lem} 
\begin{proof}
Our proof here is essentially part of the proof given in \cite[Lemma 3.2.23]{gabram}, but we include it for convenience.\\
\indent Let $\mathscr{C}$ be the complex $M \stackrel{\phi}{\longleftarrow} N$ concentrated in degrees $0$ and $1$ and consider the complex $\mathscr{C} \otimes_A B$. We have two converging spectral sequences:
\[E_{p,q}^{2,h} := \tor{q}{A}{H_{p}(\mathscr{C})}{B} \Rightarrow H_{p+q}(\text{Tot}(\mathscr{C} \otimes_A B)),\]
\[ E_{p,q}^{1,v} := \tor{q}{A}{\mathscr{C}_p}{B} \Rightarrow H_{p+q}(\text{Tot}(\mathscr{C} \otimes_A B)).\]
We immediately derive that $\coker(\phi_B)=E_{0,0}^{\infty,v} = E_{0,0}^{\infty,h} = \coker(\phi)\otimes_A B$ and the two exact sequences:
\[\tor{2}{A}{\coker(\phi)}{B} \to \ker(\phi) \otimes_A B \to H_1(\text{Tot}(\mathscr{C} \otimes_A B)),\] \[\coker(\tor{1}{A}{M}{B} \to \tor{1}{A}{N}{B}) \to H_1(\text{Tot}(\mathscr{C} \otimes_A B)) \to \ker(\phi_B).\]
So we see that $\text{Ann}_A(\coker(\phi_B))^m = \text{Ann}_A(\coker(\phi)\otimes_A B)^m \subset \text{Ann}_A(\coker(\phi))$ and from the two exact sequences:
\[(\text{Ann}_A(\ker(\phi_B)) \cdot \text{Ann}_A(\tor{1}{A}{B}{N}) \cdot \text{Ann}_A(\coker(\phi)))^m \subset \text{Ann}_A(\ker(\phi) \otimes_A B)^m \subset \text{Ann}_A(\ker(\phi)).\]
\end{proof}
Moreover, (*) comes up rather naturally in usual ring theory, and we discuss the almost variant:
\begin{lem}\label{lem:nilpotent}
Let $\phi: A \to B$ be a morphism of almost algebras that is finite with nilpotent kernel. Then, $\phi$ has property (*).
\end{lem}
\begin{proof}
This is essentially a reworded version of \cite[Lemma 3.2.23]{gabram}. First we recall the lemma used there: \cite[Lemma 3.2.21]{gabram}
\begin{lem}
Let $R$ be a ring, $M$ a finitely generated $R$-module such that $\text{Ann}_R(M)$ is a nilpotent ideal. Then $R$ admits a filtration $0 = J_0 \subset ... \subset J_m = R$ such that each $J_{i+1}/J_i$ is a quotient of a direct sum of copies of $M$.
\end{lem}
We want to apply the lemma upon applying $(-)_*$ to the given situation, but it isn't necessarily true that $B_*$ is still a finitely generated $A_*$-module; nevertheless, we may find a finitely generated $A_*$-module $Q$ such that $\ideal{m} \cdot B_* \subset Q \subset B_*$ and that $Q$ satisfies the condition of the lemma (see Remark~\ref{rem:initrem} (i) and (iv)). We thus know that there is a filtation $0 = J_0 \subset ... \subset J_m = A_*$ such that each quotient $J_{i+1}/J_i$ is a quotient of a direct sum of copies of $Q$. So, as $Q \hookrightarrow B_*$ is an almost isomorphism:
\[\text{Ann}_{A}(M \otimes_A B) = \text{Ann}_{A}(M \otimes_A Q^a) \subset \text{Ann}_A(M \otimes_A (J_{i+1}/J_{i})^a)\]
for each $i=0,1,...,m-1$, and therefore:
\[\text{Ann}_{A}(M \otimes_A B)^m \subset \prod_{i=0}^{m-1}\text{Ann}_A(M \otimes_A (J_{i+1}/J_{i})^a)\subset \text{Ann}_A(M).\]
\end{proof}

\begin{cor}\label{cor:nildescent}
Assume in addition that $A \to B$ is an epimorphism and $M$ is flat, then $M_B$ being almost projective over $B$ implies that $M$ is almost projective over $A$.
\end{cor}
\begin{proof}
We must show that $\alext{p}{A}{M}{N} =0$ for all $A$-modules $N$. Let $I$ be the kernel of $A \to B$; by the assumption, $I$ is nilpotent, so by usual devissage arguments one can can reduce to the case that $IN=0$ and therefore $N$ is a $B$-module. We then have a spectral sequence:
\[\alext{p}{B}{\tor{q}{A}{M}{B}}{N} \Rightarrow \alext{p+q}{A}{M}{N},\footnote{Here, $M$ is an $A$-module and $N$ is a $B$-module.}\]
which gives us the isomorphism $\alext{p}{B}{M \otimes_A B}{N} \simeq \alext{p}{A}{M}{N}$, so the result follows. 
\end{proof}
\begin{proof}[Proof of Lemma~\ref{thm:descent1}] See~\cite[Lemma 3.2.25]{gabram}.
\end{proof}
While Theorem~\ref{thm:descent1} provides a very general framework for when we can descend almost finite and almost finitely presented modules, we would like a set-up that descends almost projective modules. Consider the following set-up:
\[\begin{tikzcd}
A_0 \arrow[r, "f_2"] \arrow[d, "f_1"]& A_2 \arrow[d, "g_2"] \\
A_1 \arrow[r, "g_1"] & A_3
\end{tikzcd}\]
such that the square is cartesian and one of the morphisms $g_1, g_2$ is an epimorphism. Geometrically, one may view $\spec{A_0}$ as being formed by gluing $\spec{A_2}$ and $\spec{A_1}$ along a closed subscheme of one of the two; we call the diagram a \textit{gluing diagram}. Note that we have a corresponding essentially commutative diagram:
\begin{equation}\label{eq:descenteq}\begin{tikzcd}
A_0\text{-Mod} \arrow[r, "f_{2*}"] \arrow[d, "f_{1*}"]& A_2\text{-Mod} \arrow[d, "g_{2*}"] \\
A_1\text{-Mod} \arrow[r, "g_{1*}"] & A_3\text{-Mod}
\end{tikzcd}\end{equation}
where the subscripts denotes the extension of scalars functors. The theorem we want to prove is: \cite[Proposition 3.4.21]{gabram}
\begin{thm}\label{thm:descent2}
The above diagram is $2$-cartesian on the subcategory of almost projective modules and flat modules.
\end{thm}
We define the category of $\cal{D}$-modules as the $2$-fibre product $\cal{D}\text{-Mod} := A_{1}\text{-Mod} \times_{A_3\text{-Mod}} A_2\text{-Mod}$. In particular, the objects of this category are triples $(M_1,M_2, \xi)$ where $M_i$ is an $A_i$-module ($i=1,2$) and $\xi: A_3 \otimes_{A_1} M_1 \to A_{3} \otimes_{A_2} M_2$ is an $A_3$-linear isomorphism. We have a natural functor:
\[\pi: A_0\text{-Mod} \to \cal{D}\text{-Mod}\]
which takes a module $M_0 \in A_0\text{-Mod}$ and takes it to the triple $(A_1 \otimes_{A_0} M_0, A_2 \otimes_{A_0} M_0, \xi_{M_0})$ where $\xi_{M_0}: A_3 \otimes_{A_1}(A_1 \otimes_{A_0} M_0) \to A_3 \otimes_{A_2}(A_2 \otimes_{A_0} M_0)$ is a natural isomorphism arising from the commutativity $g_{1} \circ f_{1} = g_2 \circ f_2$. It's clear what this functor thus does on morphisms. We set $\cal{D}\text{-Mod}_{\text{fl}} := A_{1}\text{-Mod}_{\text{fl}} \times_{A_3\text{-Mod}_{\text{fl}}} A_2\text{-Mod}_{\text{fl}}$ and $\cal{D}\text{-Mod}_{\text{apr}} := A_{1}\text{-Mod}_{\text{apr}} \times_{A_3\text{-Mod}_{\text{apr}}} A_2\text{-Mod}_{\text{apr}}$ for the fibered product of the sub-categories of flat and almost projective modules, respectively.\\
\indent We can also construct a functor going in the other way. Indeed, given an object $(M_1, M_2, \xi)$ in $\cal{D}\text{-Mod}$, set $M_3 := A_3 \otimes_{A_2} M_2$; we have natural morphisms $M_2 \to M_3$ and $M_1 \to A_3 \otimes_{A_1} M_1 \stackrel{\xi}{\to} M_3$, so we can form the fibered product $T(M_1, M_2, \xi) := M_1 \times_{M_3} M_2$. One thus obtains a functor $T: \cal{D}\text{-Mod} \to A_0\text{-Mod}$ that is right adjoint to $\pi$, which follows from the universal property of fibered products and the natural isomorphism $\hom{A}{M}{N} \simeq \hom{B}{B \otimes_A M}{N}$ where $A \to B$ is a morphism of (almost) algebras, $M$ is an $A$-module and $N$ is a $B$-module. We set $\varepsilon: 1_{A_0\text{-Mod}} \to T \circ \pi$ and $\eta: \pi \circ T \to 1_{\cal{D}\text{-Mod}}$ to be the unit and counit of the adjunction, respectively. The next lemma is a general property of adjunctions.
\begin{lem}\label{lem:adj}
The functor $\pi$ induces an equivalence of full subcategories:
\[\{X \in \text{Ob}(A_{0}\text{-Mod}) | \varepsilon_X \text{ is an isomorphism}\} \to \{Y \in \text{Ob}(\cal{D}\text{-Mod}) | \eta_Y \text{ is an isomorphism}\}.\]
\end{lem}
\begin{lem}\label{lem:descprop1}
Let $M$ be any $A_0$-module. Then, $\varepsilon_M$ is an epimorphism with kernel $\im(\tor{1}{A_0}{M}{A_3} \to M)$. Thus if $\tor{1}{A_0}{M}{A_3}=0$, then $\varepsilon_M$ is an isomorphism.
\end{lem}	
\begin{proof}
See \cite[Lemma 3.4.10]{gabram}.
\end{proof}
\begin{lem}\label{lem:descprop2}
$\eta_{(M_1,M_2,\xi)}$ is an isomorphism for all objects $(M_1,M_2, \xi)$ in $\cal{D}$-Mod.
\end{lem}
\begin{proof}
See \cite[Lemma 3.4.14]{gabram}.
\end{proof}
Our first goal is to prove Theorem~\ref{thm:descent2} in the case of flat modules. We discuss the next series of lemmas in preparation for this.
\begin{rem}\label{rem:flatdescover}
For us, Lemmas~\ref{lem:descprop1} and \ref{lem:descprop2} will be very important; roughly speaking, they say that we can identify a flat module $M$ over $A_0$ via its `parts', and vice versa. Combining this with Lemma~\ref{lem:adj}, we reduce proving Theorem~\ref{thm:descent2}, in the case of flat modules, to showing that $T$ restricts to a functor $\cal{D}\text{-Mod}_{\text{fl}} \to A_0\text{-Mod}_{\text{fl}}$; likewise for almost projective modules. 
\end{rem}
\begin{lem}\label{lem:descreltor}
Let $M$ be any $A_0$-module and $n\ge 1$ an integer. The following conditions are equivalent:
\begin{enumerate}[(i)]
\item $\tor{j}{A_0}{M}{A_i} = 0$ for every $j \le n$ and $i=1,2,3.$
\item $\tor{j}{A_i}{A_{i} \otimes_{A_0} M}{A_3}=0$ for every $j \le n$ and $i=1,2$.
\end{enumerate}
\end{lem}
\begin{proof}
See~\cite[Lemma 3.4.15]{gabram}.
\end{proof}
The next lemma, \cite[Lemma 3.4.18]{gabram}, is crucial, and we reproduce the proof for the convenience of the reader.
\begin{lem}\label{lem:flatdesc}
Let $M$ be any $A_0$-module. We have:
\begin{enumerate}[(i)]
\item The map $A_0 \to A_1 \times A_2$ fulfils condition (*) i.e. $\text{Ann}_{A_0}( M \otimes_{A_0} A_1) \cdot \text{Ann}_{A_0}(M \otimes_{A_0} A_2) \subset \text{Ann}_{A_0}(M)$.
\item $M$ admits a three-step filtration $0 \subset F_0(M) \subset F_1(M) \subset F_2(M) = M$ such that $F_0(M)$ and $\text{gr}_2(M)$ are $A_2$-modules and $\text{gr}_1(M)$ is an $A_1$-module.
\item If $(A_1 \times A_2) \otimes_{A_0} M$ is flat over $A_1 \times A_2$, then $M$ is flat over $A_0$.
\end{enumerate}
\end{lem}
\begin{proof}Fix $I$ to be the common kernel of $A_1 \to A_3$ and $A_0 \to A_2$.
\begin{enumerate}[(i)]
\item We note first that $\text{Ann}_{A_0}(M \otimes_{A_0} A_2) \cdot M \subset IM$. Next, have the composition $I \otimes_{A_1} (A_{1} \otimes_{A_0} M) \to I \otimes_{A_0} M \to IM$ is an epimorphism, thus $\text{Ann}_{A_0}( M \otimes_{A_0} A_1) \cdot IM = 0$. Hence, $\text{Ann}_{A_0}( M \otimes_{A_0} A_1) \cdot \text{Ann}_{A_0}(M \otimes_{A_0} A_2) \subset \text{Ann}_{A_0}(M)$, so $\text{Ann}_{A_0}( M \otimes_{A_0} (A_1 \times A_2))^2 \subset \text{Ann}_{A_0}(M)$.
\item Set $F_0(M) := \ker(\varepsilon_M)$. By Lemma~\ref{lem:descprop1}, we derive that $F_0(M) \simeq \im(\tor{1}{A_0}{M}{A_3} \to M)$, which is killed by $I$ and therefore is an $A_0/I \simeq A_2$-module. We then set $F_1(M) := \varepsilon_M^{-1}(M \otimes_{A_0} A_1)$ and it is clear that this results in a valid filtration.
\item We are given that $M \otimes_{A_0} A_i$ is flat over $A_i$ for $i=1,2$, and too, by Lemma~\ref{lem:descreltor}, that $\tor{1}{A_0}{M}{A_i} = 0$ for $i=1,2,3$. The result now follows from Theorem~\ref{thm:genlocflat}.
\end{enumerate}
\end{proof}
\begin{proof}[Proof of Theorem~\ref{thm:descent2}]
In the case of flat modules, combine Remark~\ref{rem:flatdescover} and Lemma~\ref{lem:flatdesc}. Again, by Remark~\ref{rem:flatdescover}, we reduce to proving that if $A_i \otimes_{A_0} M$ is almost projective over $A_i$ for $i=1,2$, then $M$ is too over $A_0$. We already know that $M$ is flat over $A_0$, so for any $A_i$-module $N$, we have $\alext{1}{A_i}{A_i \otimes_{A_0} M}{N} = \alext{1}{A_0}{M}{N}$, as one can check by using the spectral sequence:
\[\alext{p}{B}{\tor{q}{A}{M}{B}}{N} \Rightarrow \alext{p+q}{A}{M}{N}.\]
This extends to all $A_0$-modules $N$ via Lemma~\ref{lem:filtorext}, so we're done.
\end{proof}
\begin{cor}\label{cor:descent3}
In the situation of (\ref{eq:descenteq}), the diagram is $2$-cartesian upon restricting to the sub-categories $\text{Alg}_{\text{fl}}$, $\text{\'{E}t}$ (i.e. \'{e}tale algebras) $\text{w.\'{E}t}$ (i.e. weakly \'{e}tale algebas), $\text{Alg}_{\text{afgfl}}$ (i.e. almost finitely generated flat algebras), $\text{Alg}_{\text{afpfl}}$ (i.e. almost finitely presented flat algebras), $\text{\'{E}t}_{\text{afp}}$ (i.e. almost finitely presented \'{e}tale algebras\textemdash these we will call the finite \'{e}tale algebras in the almost category).
\end{cor}
\begin{proof}
The first three follow from Theorem~\ref{thm:descent2} upon ensuring that $\pi$ and $T$ take algebras to algebras. Indeed, the claim for $\text{Alg}_{\text{fl}}$, and for the others, one further utilises that $\mu_{B_i/A_i} = \mu_{B_0/A_0} \otimes_{A_0} 1_{A_i}$ where $B_0$ is an $A_0$-algebra and $B_i := B_0 \otimes_{A_0} A_i$. For $\text{Alg}_{\text{afgfl}}$, one concludes via Theorem~\ref{thm:descent1} and Lemma~\ref{lem:flatdesc}  (i). For $\text{Alg}_{\text{afpfl}}$, suppose $M$ is an $A_0$-algebra such that $A_i \otimes_{A_0} M$ that is both flat and almost finitely presented over $A_i$ for $i=1,2$. We know that $M$ is at the very least flat, but by Theorem~\ref{thm:descent1} and Lemma~\ref{lem:flatdesc} (i), this is enough to conclude that it is also almost finitely presented. $\text{\'{E}t}_{\text{afp}}$ follows similarly.
\end{proof}
\section{Almost Witt vectors and purity}
Suppose we are given an almost set-up $(R, \ideal{m})$ that satisfies Condition (B). Our first question that we hope to answer concerns the lifting of almost set-ups to finite length Witt vectors.\footnote{It was told to the author that some of these results have been published by Gabber and Ramero in their sequel, \textit{Foundations of Almost Ring Theory}. To rectify the situation, in proofs where there was minor overlap, we have chosen to refer the reader to their book.} To fix notation, we set $W_0(A) := A$ as our indexing convention, $\omega_i$ to be the Witt polynomial of degree $p^i$, and $F$ the Witt-vector Frobenius. Before proceeding, we begin with a lemma, 
\begin{lem}\label{lem:powersubset}
Let $R$ be a ring with a ideal $I \subset R$ satisfying Condition (B). Suppose that we have a subset $S \subset I$ of elements for which $I = \sum_{x \in S} Rx$, and that $S$ is both additively and multiplicatively closed. Then, $I = \sum_{x \in S} Rx^k$ for any $k > 1$.
\end{lem}
\begin{proof}
One replicates the argument given in \cite[Claim 2.1.9]{gabram}.
\end{proof}
\begin{lem}\label{lem:nilpotentlift}
Let $f: S \to R$ be a surjective map of rings with nilpotent kernel. Then there is a unique ideal lifting $\ideal{m}_S$ of $\ideal{m}$ such that $\ideal{m}_S^2 = \ideal{m}_S$. Furthermore, $\ideal{m}_S$ satisfies Condition (B) if $\ideal{m}$ does.
\end{lem} 
\begin{proof}
See~\cite[Lemma 14.7.1]{gabram2}.
\end{proof}
The following theorem will be used to define a `unique' almost set-up on finite length Witt vectors.
\begin{thm}\label{thm:wittlift}
Let $(R, \ideal{m})$ be an almost set-up, with $\ideal{m}$ satsifying Condition (B). Then there exist ideals $\ideal{m}_n \subset W_n(R)$ such that:
\begin{enumerate}[(i)]
\item $\ideal{m}_n^2 = \ideal{m}_n$, and hence $(W_n(R), \ideal{m}_n)$ is too an almost set-up.
\item $\omega_n(\ideal{m}_n)R = \ideal{m}$ and $\text{pr}_n(\ideal{m}_n) = \ideal{m}_{n-1}$ where $\text{pr}_n: W_n(R) \to W_{n-1}(R)$ is the canonical projection.
\end{enumerate}
\end{thm}
We have two maps $\omega_n: W_{n}(R) \to R$ and $\text{pr}_{n}: W_n(R) \twoheadrightarrow W_{n-1}(R)$. Denote by $\alpha_n$ the induced map $W_n(R) \to W_{n-1}(R) \times R$, and let $\overline{W_n}(R):= \im(\alpha_n)$. We have a gluing diagram:
\begin{equation}\begin{tikzcd}
\overline{W_n}(R) \arrow[r, "\text{pr}_n"] \arrow[d, "\omega'_n"] & W_{n-1}(R)\arrow[d, "\overline{\omega_n}"] \\
R \arrow[r, twoheadrightarrow, "\pi"] &R/p^n
\end{tikzcd}\end{equation}
One may check the kernel $I_R := \ker(\alpha_n)$ is square zero i.e. $I_R^2 = 0$. \cite[Proposition 8.1 (b)]{borger2015basic} Our first step to proving Theorem~\ref{thm:wittlift} is to lift $\ideal{m}$ to $\overline{W_n}(R)$. To do so, we work more broadly.
\begin{thm}\label{thm:gluinglift}
Suppose we are given a gluing diagram:
\[\begin{tikzcd}
A_0 \arrow[r, "f_1"] \arrow[d, "f_2"] & A_1 \arrow[d, "g_1"]\\ A_2 \arrow[r, "g_2"] & A_3
\end{tikzcd}\]
with $g_2$ (and hence $f_1$) surjective, and ideals $\ideal{m}_1 \subset A_1$, $\ideal{m}_2 \subset A_2$ such that they agree over $A_3$ i.e. $g_1(\ideal{m}_1) A_3 = g_2(\ideal{m}_2)$. Then there exists a unique ideal $\ideal{m} \subset A_0$ with the following properties:
\begin{enumerate}[(i)]
\item $\ideal{m}^2 = \ideal{m}$, and hence $(A_0, \ideal{m})$ is an almost set-up.
\item $f_1(\ideal{m})A_1 = \ideal{m}_1$ and $f_2(\ideal{m})A_2 = \ideal{m}_2$.
\end{enumerate}
Furthermore, if both $\ideal{m}_1$ and $\ideal{m}_2$ satisfy Condition (B), then so does $\ideal{m}$.
\end{thm}
\begin{proof}
For the proofs we refer the reader to \cite[Proposition 14.7.5]{gabram2}. We also present a slightly different approach to the last statement: it suffices to show that the ideals $\ideal{m}'_{k} \subset \ideal{m}$ generated by the $k^{\text{th}}$-powers of $\ideal{m}$ satisfies condition (ii). It is indeed clear that $f_1(\ideal{m}'_{k}) = \ideal{m}_1$ because $\ideal{m}_1$ has Condition (B). Since $f_2(\ideal{m})A_2 = \ideal{m}_2$, by Lemma~\ref{lem:powersubset}, it is also clear that $f_2(\ideal{m}'_k)A_2 = \ideal{m}_2$.
\end{proof}
\begin{proof}[Proof of Theorem~\ref{thm:wittlift}]We proceed by induction on $n$, where the base case $n=0$ is tautological. Now, assuming the existence of such lift $\ideal{m}_{n-1} \subset W_{n-1}(R)$, we shall construct $\ideal{m}_n$. Recall the gluing diagram:
\[\begin{tikzcd}
\overline{W_n}(R) \arrow[r, "\text{pr}_n'"] \arrow[d, "\omega'_n"] & W_{n-1}(R)\arrow[d, "\overline{\omega_n}"] \\
R \arrow[r, twoheadrightarrow, "\pi"] &R/p^n
\end{tikzcd}\]
We want to show that $\pi(\ideal{m}) R/p^n = \overline{\omega_n}(\ideal{m}_{n-1}) R/p^n$. Their reductions modulo $p$ are equal because $\ideal{m}_{n-1}$ satisfies Condition (B). Since both ideals are idempotent, and $p$ is nilpotent in $R/p^n$, we conclude equality in $R/p^n$ by Lemma~\ref{lem:nilpotentlift}.\\
\indent We can then apply Theorem~\ref{thm:gluinglift} to obtain an ideal $\overline{\ideal{m}_{n}}$ such that  $\omega_n(\overline{\ideal{m}_n})R = \ideal{m}$ and $\text{pr}_n(\overline{\ideal{m}_n}) = \ideal{m}_{n-1}$. Since $I_R$ is nilpotent, the ideal $\overline{\ideal{m}_n} \subset \overline{W_n}(R)$ lifts to an ideal $\ideal{m}_n \subset W_n(R)$ satisfying the desired conditions.
\end{proof}
All in all, given an almost set-up $(R, \ideal{m})$, we have constructed almost set-ups $(W_n(R), \ideal{m}_n)$. We want to show that $W_n$ preserves almost isomorphisms. Given $R$-algebras $A, B$ and an almost isomorphism $f: A \to B$, we hope the kernel and cokernel of $W_nf$ are killed by $\ideal{m}_{n}$. For the kernel, by induction, this reduces to $V_n(\ker(f))/V_{n+1}(\ker(f)) \simeq \ker(f)$ being killed by $\omega_n(\ideal{m}_n) = \ideal{m}$; clear. A similar argument applies for the cokernel, and therefore, $W_n$ becomes a well defined endofunctor in the category of $\cal{B}^a$-Alg (notation of \ref{fibalm}). We also note that $W_n(A) = W_n(A_*)^a$, which will be used throughout this section without further notice.
\begin{rem}If we were primarily interested in ideals such as $\ideal{m} = (f^{1/p^{\infty}})$, then the whole discussion above simplifies greatly. 
\end{rem}
\indent Condition (B) gives $\omega_i(\ideal{m}_n) A = \ideal{m}$ and $F(\ideal{m}_n) A= \ideal{m}$, so the morphisms are defined in $\cal{B}$ and give natural transformations between $W_n(-)$ and the identity, $W_{n-1}(-)$ respectively. These results allow us to consider these morphisms in the almost category, and they behave as expected. For a morphism of rings $f: A \to B$ and a $B$-module $M$, we will use the notation $f_*(M)$ for when we want to view $M$ as an $A$-module. For an arbitrary $R^a$-algebra $A$, set $A_0:=A/pA$, and for $\mathbb{F}_p$-algebras $A_0 \to B_0$, let $\Phi_{B_0/A_0}$ denote the relative Frobenius. For an $A$-algebra $B$, set $V_n(B):=V_n(B_*)^a$, where it is straight forward to see the isomorphism $V_i(A)/V_{i+1}(A) \simeq \omega_{i*}(A)$.
\begin{thm}\label{thm:wittdescent}
Let $A \to B$ be a flat morphism of $R^a$-algebras, and suppose that $\Phi_{B_0/A_0}$ is an isomorphism. Then,  $W_n(f): W_n(A) \to W_n(B)$ is flat and $\omega_{i*}(B) \simeq W_n(B) \otimes_{W_n(A)} \omega_{i*}(A)$ for each $0 \le i \le n$.
\end{thm}
Let us examine the immediate effects of this:
\begin{cor}
Suppose $f: A \to B$ is weakly \'{e}tale, then $W_n(B) \otimes_{W_n(A)} W_n(C) \simeq W_n(B \otimes_A C)$ for $C$ an $A$-algebra, and $W_n(f): W_n(A) \to W_n(B)$ is weakly \'{e}tale.
\end{cor}
\begin{proof}
Note since $A \to B$ is weakly \'{e}tale, therefore so is $A_0 \to B_0$, and in particular $\Phi_{B_0/A_0}$ is an isomorphism. To prove the first proposition, by considering the exact sequences $0 \to V_{i+1}(C) \to V_i(C) \to w_{i*}(C) \to 0$, we reduce by induction to proving the following isomorphism of $W_n(A)$-modules:
\[W_n(B) \otimes_{W_n(A)} \omega_{i*}(C) \simeq \omega_{i*}(B) \otimes_{\omega_{i*}(A)} \omega_{i*}(C).\]
Noting that $\omega_{i*}(C)$ is an $\omega_{i*}(A)$-module, by base-change this reduces to the second part of Theorem~\ref{thm:wittdescent}. To prove $W_n(f)$ is weakly \'{e}tale, we have that the multiplication map $\mu_B: B \otimes_A B \to B$ is flat, but in particular also weakly \'{e}tale too. Therefore by the first part of this proposition, $W_n(\mu_B): W_n(B) \otimes_{W_n(A)} W_n(B) \to W_n(B)$ is flat, and the conclusion follows.
\end{proof}
\begin{cor}\label{cor:wittpushoutfrob}
Assume the conditions of Theorem~\ref{thm:wittdescent}; the diagram:
\[\begin{tikzcd}W_n(A) \arrow[r, "F"] \arrow[d, "W_n f"] & W_{n-1}(A) \arrow[d, "W_{n-1}f"] \\
W_n(B) \arrow[r, "F"] & W_{n-1}(B) \end{tikzcd}\]
is cocartesian i.e. there exists an isomorphism $W_n(B) \otimes_{W_n(A)} W_{n-1}(A) \to W_{n-1}(B)$ of $W_{n-1}(A)$-modules.
\end{cor}
\begin{proof}
We proceed via induction on $n$; the base case where $n=1$ asserts the isomorphism $W_1(B) \otimes_{W_1(A), F} A$, which reduces to the second part of Theorem~\ref{thm:wittdescent} by noticing that $F$ and $w_1$ coincide as maps from $W_1(A)$ to $A$. Now, consider the exact sequence of $W_{n+1}(A)$-modules:
\[\begin{tikzcd}
0 \arrow[r] & V_{n}(A) \arrow[r] & W_{n}(A) \arrow[r, "\text{pr}_n"] & W_{n-1}(A) \arrow[r] & 0
\end{tikzcd}\]
The action on the middle module is via $F: W_{n+1}(A) \to W_{n}(A)$, and the action on the right module is via $F \circ \text{pr}_{n+1} = \text{pr}_{n} \circ F$. The induced action on the ideal $V_n(A)$ is by $\omega_n \circ F = \omega_{n+1}$, so we identify $V_n(A) \simeq \omega_{n+1*}(A)$ as $W_{n+1}(A)$-modules. Now, we may tensor the sequence above with the flat $W_{n+1}(A)$-module $W_{n+1}(B)$ to obtain the exact sequence:
\[\begin{tikzcd}
0 \arrow[r] & V_{n}(B) \arrow[r] & W_{n+1}(B) \otimes_{W_{n+1}(A), F} W_{n}(A) \arrow[r, "\text{pr}_n"] & W_{n-1}(B) \arrow[r] & 0
\end{tikzcd}\]
It then directly follows that the natural map $W_{n+1}(B) \otimes_{W_{n+1}(A), F} W_{n}(A) \to W_{n}(B)$ is an isomorphism.
\end{proof}
The relative Frobenius condition appearing in Theorem~\ref{thm:wittdescent} is not usually seen in the literature. As one can see through the proof of Theorem~\ref{thm:wittdescent2}, the main obstruction to proving that $W_n$ preserves flat algebras is showing the claim $w_{n*}(A) \otimes_{W_{n}(A)} W_n(B) \simeq w_{n*}(B)$. When $p=0$ in $A$, this simply reduces to the relative Frobenius condition. Our proof of Theorem~\ref{thm:wittdescent} actually shows that the claim is independent of proving flatness, mostly due to the following lemma.
\begin{lem}\label{lem:nilflatiso}
Let $f: M \to N$ be a morphism of $R$-modules, with $N$ being flat. Let $S$ be an $R$-algebra satisfying Condition (*). If $f \otimes_R S$ is an isomorphism, then so is $f$.
\end{lem}
\begin{proof}
We begin with the following fact: if an $R$-module $M$ satisfies $M\otimes_R S \simeq 0$ then $M \simeq 0$. Indeed, one obtains that $R=\text{Ann}_R(M \otimes_R S)^m \subseteq \text{Ann}_R(M) \subseteq R$ for some $m \in \mathbb{N}_{>0}$, so $M \simeq 0$.\\
\indent Using this fact, in combination with the right exactness of tensoring, one concludes that $\coker(f) \simeq 0$, or in other words, that $f$ is an epimorphism. Thus, we have a short exact sequence $0 \to \ker(f) \to M \to N \to 0$. Applying $-\otimes_R S$ and using Tor sequences, we see that $0 \to \ker(f) \otimes_R S \to M\otimes_R S \to N \otimes_R S \to 0$ is exact, so we know that $\ker(f) \otimes_R S \simeq 0$. Using the fact at the beginning, one concludes that $\ker(f) \simeq 0$ and, thusly, that $f$ is an isomorphism. 
\end{proof}
\begin{lem}\label{lem:annwitt}
Suppose that $\Phi_{B_0/A_0}$ is an epimorphism and $B$ is flat over $A$. Then, $I_A \otimes_{W_n(A)} W_n(B) \to I_B$ is an epimorphism.\footnote{Recall that $I_R:=\ker(W_n(R) \to \overline{W_n}(R))$, and note that $I_R \simeq \text{Ann}_{R}(p^n)$} \footnote{The author would like to thank Dr James Borger for helping write their initial proof in a more readable way.}
\end{lem}
\begin{proof}
Note that we have an isomorphism $\text{Ann}_{A}(p^n) \otimes_A B \stackrel{\text{ev}}{\to} \text{Ann}_B(p^n)$ as $B$ is flat over $A$; hence $I_A \otimes_A B \simeq I_B$. As $I_A \otimes_A-$is right exact, it suffices to show that $w_{n*}(A) \otimes_{W_n(A)} W_n(B) \to B$ is an epimorphism of $A$-modules. \\
\indent Define a bifunctor by the rule $- \otimes W_n(-) := ((-)_* \otimes_{\mathbb{Z}} W((-)_*))^a$. Consider the morphism $\phi_{n}: A \otimes W_n(B)) \to B$ which takes $a \otimes b \mapsto aw_n(b)$ for almost elements  $a \in A_*, b \in W_n(B_*)$. We have a factoring $A \otimes W_n(B) \twoheadrightarrow w_{n*}(A) \otimes_{W_n(A)} W_n(B) \to B $ and therefore it suffices to show that $\phi_n$ is an epimorphism.\\
\indent We proceed by induction on $n$; the case when $n=0$ is tautological.  Consider the diagram with exact rows:
\[\begin{tikzcd}
 A \otimes W_{n-1}(B) \arrow[r, "\text{id}_A\otimes V"] \arrow[d, "\phi_{n-1}"]& A \otimes W_n(B) \arrow[r] \arrow[d, "\phi_{n}"] & B \otimes A \arrow[r] \arrow[d, "b \otimes a \mapsto b^{p^n} a"]& 0\\
B \arrow[r, "p"] & B \arrow[r] & B/pB \arrow[r] & 0
\end{tikzcd}\]
Commutativity of the left-hand square follows from the identity $w_n \circ V = pw_{n-1}$. The left-down arrow is an epimorphism by induction, and the right-down arrow is an epimorphism by the assumption on the relative Frobenius. We then conclude the middle-down arrow is an epimorphism, as desired.
\end{proof}
\begin{proof}[Proof of Theorem~\ref{thm:wittdescent}]
Recall the gluing diagram:
\begin{equation}\label{eq:wittdescent}\begin{tikzcd}
\overline{W_n}(A) \arrow[r, "\text{pr}_n'"] \arrow[d, "\omega'_n"] & W_{n-1}(A)\arrow[d, "\overline{\omega_n}"] \\
A \arrow[r, twoheadrightarrow, "\pi"] & A/p^n
\end{tikzcd}\end{equation}
Let $I_A := \ker(\alpha_n)$, which may be characterised as the $W_n(A)$-module $\text{Ann}_{\omega_{n*}(A)}(p^n)$. As an ideal of $W_n(A)$, one may check that $I_A^2=0$.\\
We will prove the statement that $W_{n}(B)$ is flat over $W_n(A)$ and $\omega_{n*}(A) \otimes_{W_n(A)} W_n(B) \simeq \omega_{n*}(B)$ via induction on $n$; the base case when $n=0$ is tautological. Assume the statement is true for $k=n-1$, we will show it then for $n$.\\

\noindent\textit{Claim:} $(W_{n-1}(B), B)$ is a gluing datum for Diagram~\ref{eq:wittdescent}.\\
We must show that $W_{n-1}(B) \otimes_{W_{n-1}(A)} \overline{\omega_n}_{*}(A/p^n) \to \overline{\omega_n}_{*}(B/p^n)$ is an isomorphism. Since $B/p^n$ is a flat $A/p^n$-module, and $p$ is nilpotent in $A/p^n$, we are in the situation of Lemma~\ref{lem:nilflatiso} (see Lemma~\ref{lem:nilpotent}) where we reduce to showing the isomorphism after tensoring with $- \otimes_{\mathbb{Z}} \mathbb{F}_p$. Thus, we need to show that $W_{n-1}(B) \otimes_{W_{n-1}(A)} \overline{\omega_n}_{*}(A_0) \simeq B_0$. The map $W_{n-1}(A) \stackrel{\overline{\omega_n}}{\to} A_0$ factors via $W_{n-1}(A) \stackrel{\omega_{n-1}}{\to} A \stackrel{\pi}{\to} A_0 \stackrel{\Phi_{A_0}}{\to}A_0$, so we have:
\begin{align*}W_{n-1}(B) \otimes_{W_{n-1}(A)} \overline{\omega_n}_{*}(A_0) &\simeq (W_{n-1}(B) \otimes_{W_{n-1}(A)} \omega_{n-1*}(A)) \otimes_{A} \overline{\pi}_{*}(A_0) \otimes_{A_0} \Phi_{A_0*}(A_0)\\ &\simeq \omega_{n-1*}(B_0) \otimes_{A_0} \Phi_{A_0*}(A_0)
\simeq \overline{\omega_{n}}_*(B_0),
\end{align*}
where the last isomorphism followed from our hypothesis on $\Phi_{B_0/A_0}$.
$\Box$\\

Hence, according to our claim, setting $A_1 := W_{n-1}(A), A_2 := A, A_3 := A/p^n$ and $\xi: W_{n-1}(B) \otimes_{W_{n-1}(A)} A/p^n \simeq B \otimes_{A} A/p^n$ in Equation~\ref{eq:descenteq}, we can form the $\overline{W}_n(A)$-module $T(W_{n-1}(B), B, \xi) = \overline{W_n}(B)$. From Lemma~\ref{lem:descprop2} we immediately derive that $\overline{W_n}(B) \otimes_{\overline{W_n}(A)}w'_{n*}(A) \simeq w'_{n*}(B)$ and that $\overline{W}_n(B)$ is a flat $\overline{W}_n(A)$-module.\\
\indent By utilising Lemma~\ref{lem:annwitt}, we obtain that $I_A W_n(B) \simeq I_B$. Therefore, $W_n(B)\otimes_{W_n(A)} \overline{W_n}(A) \simeq \overline{W_n}(B)$. Since $\omega_n = \alpha_n \circ \omega'_n$, we have:
\[W_n(B) \otimes_{W_n(A)} \omega_{n*}(A) \simeq (W_n(B) \otimes_{W_n(A)} \overline{W_n}(A)) \otimes_{\overline{W_n}(A)} \omega'_{n*}(A) \simeq \omega_{n*}(B).\]
Lastly, to verify that $W_n(B)$ is flat over $W_n(A)$, we appeal to Corollary~\ref{cor:almlocflat}, where it suffices to show that $I_A \otimes_{W_n(A)} W_n(B) \simeq I_B$, or indeed that 
\[\text{Ann}_{\omega_{n*}(A)}(p^n) \otimes_{W_n(A)} W_n(B) \simeq \text{Ann}_{\omega_{n*}(B)}(p^n).\]
But of course, since $I_A^2 = 0$, it is an $\overline{W_n}(A)$-module, so we simply must check that:
\[\text{Ann}_{\omega'_{n*}(A)}(p^n) \otimes_{\overline{W_n}(A)} \overline{W_n}(B) \simeq \text{Ann}_{\omega'_{n*}(B)}(p^n)\]
which follows from the flatness of $\overline{W_n}(B)$ over $\overline{W_n}(A)$. This completes the induction step. We are only left to show that $\omega_{i*}(A) \otimes_{W_n(A)} W_n(B) \simeq \omega_{i*}(B)$. But this follows from the isomorphism $\omega_{i*}(A) \otimes_{W_n(A)} W_n(B) \simeq \omega_{i*}(A) \otimes_{W_{i}(A)} W_{i}(B) $ and induction.
\end{proof}
\begin{thm}\label{thm:wittdescent2}
Assuming the conditions of Theorem~\ref{thm:wittdescent}, if we also suppose that $A \to B$ is almost finitely presented, almost finitely generated, or almost projective, then so is $W_n(A) \to W_n(B)$.
\end{thm}
\begin{proof}
We know that the map $W_n(A) \twoheadrightarrow \overline{W_n}(A)$ is an epimorphism with nilpotent kernel and thus satisfies Condition (*) in Theorem~\ref{thm:descent1}. By utilising Theorem~\ref{thm:wittdescent} together with Theorem~\ref{thm:descent1} and Corollary~\ref{cor:nildescent}, we reduce the statement to $\overline{W_n}(B)$ over $\overline{W_n}(A)$. By utilising Equation~\ref{eq:wittdescent}, Theorem~\ref{thm:wittdescent}, Theorem~\ref{thm:descent2} its Corollary~\ref{cor:descent3}, we reduce the statement to $W_{n-1}(B)$ as a $W_{n-1}(A)$-module, then by induction to $n=0$ where that statement is tautological.
\end{proof}
Finally, we get:
\begin{cor}\label{cor:wittmain}
If $A \to B$ is weakly \'{e}tale (resp. \'{e}tale, almost finite \'{e}tale, almost finitely presented \'{e}tale), then $W_n(A) \to W_n(B)$ is weakly \'{e}tale (resp. \'{e}tale, almost finite \'{e}tale, almost finitely presented \'{e}tale).
\end{cor}
\begin{proof}
By Theorem~\ref{thm:wittdescent2}, it is sufficient to show this for when $A \to B$ is (weakly) \'{e}tale. All we need to show is that $W_n(B)$ is (weakly) unramified over $W_n(A)$, which is to say, for the multiplication map $\mu_{W_n(B)/W_n(A)}: W_n(B) \otimes_{W_n(A)} W_n(B) \to W_n(B)$, $W_n(B)$ is required to be an (flat) almost projective $W_n(B) \otimes_{W_n(A)} W_n(B)$-module. But this is clear, because we know that $W_n(B) \otimes_{W_n(A)} W_n(B) \simeq W_n(B\otimes_A B)$, and $B \otimes_A B \to B$ is (flat) almost projective, so Theorem~\ref{thm:wittdescent2} gives the result.
\end{proof}
We end this section with a modest application. In what follows, we will use the adjective `almost' for a property of modules or algebras in the usual category to be a statement for their almostification. We say a ring $A$ is Witt perfect if $F: W_n(A) \to W_{n-1}(A)$ is surjective. 
\begin{cor}
If $A \to B$ is (almost) weakly \'{e}tale and $A$ is Witt perfect, then $B$ is almost Witt perfect. In the situation when $\ideal{m} = (p^{1/p^{\infty}})$, if $B$ is further assumed to be integrally closed in $B[1/p]$, then $B$ is Witt perfect.
\end{cor}
\begin{proof}
The first part follows directly from Corollary~\ref{cor:wittpushoutfrob}. It now suffices to show the second part. Assuming that $\ideal{m} = (p^{1/p^{\infty}})$, and that $B$ is integrally closed in $B[1/p]$, we must show that $F: W_n(B) \to W_{n-1}(B)$ is surjective. At the very least, we know that it is almost surjective, i.e. the cokernel is killed by $(p^{1/p^{\infty}})$. We now refer to \cite{kedlaya}, where proving the surjectivity of $F: W_{n}(B) \to W_{n-1}(B)$ is equivalent to doing it when $n=1$. Set $c=1-1/p$. For any $b \in B$, we may write $p^{1/p^2}b = b_0^p + pb_1$, and similarly $p^{1/p}b_1 = b_2^p + pb_3$, giving 
\[p^{1/p^2}b = b_0^p +p^{c}b_2^p + p^{1 + c}b_3 = x_0'^{p} + p^{1+c/p}x_1\]
where $x_0' = b_0 + p^{c/p}b_2$ and $x_1 \in B$. We then obtain that $x_0=x_0'/p^{1/p^3}$ solves the polynomial:
\[X^p - p^{1 + \frac{c}{p}-\frac{1}{p^2}}x_1 - b \in B[X].\]
Since $B$ is integrally closed in $B[1/p]$, we then obtain that $x_0 \in B$, and that $b-x_0^p \in pB$.
\end{proof}
\begin{rem}
\begin{enumerate}[(i)]
\item One can check that the condition that $A \to B$ is almost weakly \'{e}tale in the previous corollary could be relaxed to flat and $\phi_{B_0/A_0}$ being an almost isomorphism; see Theorem~\ref{thm:wittdescent}. 
\item It would be interesting to know whether the previous corollary could be specialised to perfectoid rings, the latter notion for our purposes being that the projection $\lim_{F} W_{n}(A) \to A$ has principal kernel in addition to $A$ being Witt-perfect.
\end{enumerate}
\end{rem}
\subsection{Almost purity}
We first deal with almost purity in characteristic $p$ following \cite{gabram}. Before getting into the discussion, let us consider a ring $R$, a non-zero divisor $t$ such that $\ideal{m} = (t^{1/p^{\infty}}) \subset R$. Note that $\ideal{m}^2 = \ideal{m}$, for the reason that $\ideal{m} \subset \ideal{m}^p$, and in this case, $\ideal{m}$ is flat as it is a colimit of free modules i.e. $\ideal{m} = \inj_{i} t^{1/p^{i}} R$ with the coresponding transition maps. Thus, $\tilde{\ideal{m}} \simeq \ideal{m}$, and we note that for an almost $R$-module $M^a$ such that $M$ is $t$-torsion-free:
\begin{align*}(M^a)_{*} = \hom{R}{\ideal{m}}{M} = \hom{R}{\inj_{i} t^{1/p^{i}}R}{M} &= \lim_i \hom{R}{t^{1/p^{i}}R}{M} \\
&= \{m \in M[1/t]: t^{1/p^i}m \in M, \forall i \ge 0\}.
\end{align*}

\begin{thm}[Theorem 3.5.28, Gabber-Ramero]\label{thm:almpurcharp}
Let $R$ be a perfect $\mathbb{F}_p$-algebra, and fix an almost set-up $\ideal{m} = (t^{1/p^{\infty}})$. We set $R^{a}\text{-f\'{E}t}$ as the category of almost finite presented \'{e}tale $R^a$-algebras. The functor $F$:
\[R^{a}\text{-f\'{E}t} \to R[1/t]\text{-f\'{E}t}: \; \; \; A \to A_{*}[1/t]\]
is an equivalence of categories.
\end{thm}

We invite the reader to read the proof given in \cite{gabram}, upon which it should be clear that existence of coperfection in characteristic $p$, which helps to `almost' characterise the integral closure of $R$ in $B$, combined with the existence of Frobenius makes almost purity not too difficult in this context. Each of these do not exist in the characteristic 0, at least in the strict sense. One take-away from the work of Bhatt and Scholze in their paper on prismatic cohomology \cite{prism} is that coperfection should be replaced by a suitable `perfectoidisation' in characteristic $0$, where one ends up with a perfectoid ring instead of a perfect ring. The Frobenius in characteristic $p$ could then be accessed in characteristic $0$ via Frobenius lifts and delta rings. All in all, the ideas in characteristic $p$ seem to, rather magically, apply more broadly.\\

Let us now prove the main theorem of this paper. Fix a perfectoid valuation ring $R$ of rank one with an element $\varpi$ admitting a compatible system of $p$-power roots.
\begin{thm}[Almost purity in Characteristic 0]\label{thm:almpurchar0}
Let $S$ be a finitely presented $R$-algebra, and suppose $S[1/p]$ is a finite \'{e}tale algebra over $R[1/p]$. Then $\cal{O}_S$, the integral closure of $R$ in $S[1/p]$, is almost finitely presented \'{e}tale over $(R,(\varpi^{1/p^{\infty}}))$.
\end{thm}
To proceed, we begin with a few lemmas. Recall that $\lim_{F} W_n(R) \twoheadrightarrow R$ has principally generated kernel~\cite[Definition 16.3.1]{gabram2}, which we denote by $\varepsilon$.
\begin{lem}\label{lem:absabh} Define $R^{\flat}:= \lim_{\Phi} R/p = (\lim_{F} W_{n}(R))/p$, and let $\overline{R^{\flat}}$ be the absolute integral closure of $R^{\flat}$. Then, $\overline{R^{\flat}}$ is almost weakly \'{e}tale over $R^{\flat}$ and $\overline{R}:=W(\overline{R^{\flat}})/\varepsilon$ is absolutely integrally closed over $R$. Hence, $\overline{R}/p^n$ is almost weakly \'{e}tale over $R/p^n$ for each $n \ge 1$.
\end{lem}
\begin{proof}
By Theorem~\ref{thm:almpurcharp}, we know that $\overline{R^{\flat}}$ is a colimit of almost weakly \'{e}tale algebras, and therefore is almost weakly \'{e}tale (Corollary~\ref{cor:colimwet}). For the second claim, we refer the reader to \cite[Proposition 3.8]{perfectoid}. For the third claim, we utilise Corollary~\ref{cor:wittmain} to conclude that $\overline{R}/p^n \simeq W_n(\overline{R^{\flat}})/\varepsilon$ is almost weakly \'{e}tale over $W_n(R^{\flat})/\varepsilon \simeq R/p^n$.
\end{proof}
\begin{lem}
$\cal{O}_S$ is uniformly almost finitely presented with uniform rank $[S[1/p]:R[1/p]]$.
\end{lem}
\begin{proof}
See~\cite[Chap 6.3.6]{gabram}.
\end{proof}
\begin{proof}[Proof of Theorem~\ref{thm:almpurchar0}]
First, we know that $\cal{O}_S$ is a Pr\"{u}fer domain, so $\overline{R}$ is faithfully flat over $\cal{O}_S$ due to the fact that $\overline{R}$ is torsion-free and integral over $\cal{O}_S$.\footnote{For our purposes Pr\"{u}fer domain is a commutative domain whose localisations at all prime ideals is a valuation ring; see~\cite[Chapter VI, Proposition 8.7.9]{nicb}. One can then see that flat modules over a Pr\"{u}fer domain are equivalent to torsion-free ones by localising at the maximal ideals and using this characterisation for valuation rings.} Therefore, for each $n \ge 1$, we conclude that $\overline{R}/p^n$ is faithfully flat over $\cal{O}_S/p^n$. By (viii) of Proposition~\ref{prop:assortlem}, we get that $\cal{O}_S/p^n$ is almost weakly \'{e}tale over $R/p^n$ for each $n$, and combined with the previous lemma and Theorem~\ref{thm:presented}, we conclude that $\cal{O}_S/p^n$ is almost finitely presented \'{e}tale over $R/p^n$ for each $n\ge 1$, for which one obtains via \cite[Theorem 5.3.27]{gabram} that $\cal{O}_S$ is almost finitely presented \'{e}tale over $R$, as desired.
\end{proof}

\bibliography{bip}

\begin{thebibliography}{10}

\bibitem{prism}
B.~Bhatt and P.~Scholze.
\newblock Prisms and prismatic cohomology.
\newblock 2017.

\bibitem{borger2015basic}
J.~Borger.
\newblock The basic geometry of witt vectors, i: The affine case.
\newblock 2015.

\bibitem{nicb}
N.~Bourbaki.
\newblock {\em Commutative Algebra}.
\newblock Hermann, 1972.

\bibitem{gerd}
G.~Faltings.
\newblock $p$-adic hodge theory.
\newblock {\em J. Amer. Math. Soc}, (1):255--299, 1998.

\bibitem{gabram}
O.~Gabber and L.~Ramero.
\newblock Almost ring theory.
\newblock 2002.

\bibitem{gabram2}
O.~Gabber and L.~Ramero.
\newblock Foundations of almost ring theory.
\newblock 2020.

\bibitem{algebragel}
S.~I. Gelfand and Yu. Manin.
\newblock {\em Homological Algebra}.
\newblock Springer, 2009.

\bibitem{kedlaya}
K~S. Kedlaya and C.~Davis.
\newblock On the witt vector frobenius.
\newblock 2014.

\bibitem{perfectoid}
P.~Scholze.
\newblock Perfectoid spaces.
\newblock 2011.

\bibitem{stacks}
The {Stack Project Authors}.
\newblock \textit{Stacks Project}, 2018.

\bibitem{tatepdiv}
J.~Tate.
\newblock p-divisible groups.
\newblock {\em Inv. Math.}, 2:105--154, 1967.

\bibitem{weibel}
C.~Weibel.
\newblock {\em Homological Algebra}.
\newblock Cambridge University Press, 1994.

\end{thebibliography}
\bibliographystyle{plain}
\end{document}